\documentclass[12pt,a4paper]{article}

\usepackage{amsmath, amsfonts, amssymb}
\usepackage{theorem}
\let\BFseries\bfseries\def\bfseries{\BFseries\mathversion{bold}}

\newtheorem{thm}{Theorem}[section]
\newtheorem{lemma}[thm]{Lemma}
\newtheorem{prop}[thm]{Proposition}

\newtheorem{cor}[thm]{Corollary}

\newtheorem{rem}[thm]{Remark}

\def\be#1\ee{\begin{equation}#1\end{equation}}
\newcommand{\bea}{\begin{eqnarray}}
\newcommand{\eea}{\end{eqnarray}}
\newcommand{\beaa}{\begin{eqnarray*}}
\newcommand{\eeaa}{\end{eqnarray*}}
\newcommand{\bei}{\begin{itemize}}
\newcommand{\eei}{\end{itemize}}
\newcommand{\bee}{\begin{enumerate}}
\newcommand{\eee}{\end{enumerate}}

\def\A{\mathcal{A}}

\def\C{\mathbb{C}}
\def\D{{\Delta}}

\newcommand{\eps}{\varepsilon}

\def\unit{{{\mathbf 1}}}

\def\H{\mathcal{H}}

\def\ph{{\varphi}}

\def\s{\sigma}

\def\xn{X_\nu}
\def\dxn{\tilde X_\nu}

\def\P{{\mathbb{P}}}

\def\R{\mathbb{R}}
\def\E{\mathbb{E}}

\def\Z{{\mathbb Z}}

\newcommand{\Imag}{\operatorname*{Im}}

\newenvironment{proof}[1][] {\noindent {\bf Proof#1:} }{\hspace*{\fill}$\square$\medskip\par}

\begin{document}
\title{Small Deviations of Smooth Stationary\\ Gaussian Processes}

\author{F. Aurzada, I.A. Ibragimov, M.A. Lifshits, and J.H. van Zanten
   } \date{\today}  \maketitle
\bigskip

\begin{abstract}
\noindent We investigate the small deviation probabilities  of a
class of very smooth stationary Gaussian processes playing an
important role in Bayesian statistical inference. Our calculations
are based on the appropriate modification of the entropy method
due to Kuelbs, Li, and Linde as well as on classical results about
the entropy of classes of analytic functions. They also involve
Tsirelson's upper bound for small deviations and shed some light
on the limits of sharpness for that estimate.
\end{abstract}

\section{Introduction}
\setcounter{equation}{0}

Let $X(t)$ be a centered stationary Gaussian process identified by its spectral
measure $F(du)$. We restrict $X$ on the interval $[0,1]$ and evaluate its small
deviations with respect to the uniform norm $||\cdot||_\infty$
in terms of the small deviation function
\[
     \ph(X, r)=-\log\P(||X||_\infty\leq r), \qquad r\to 0.
\]
See \cite{LiSha4}, \cite{Lif1} for many motivations for the study of small deviations and
\cite{Lif2} for a complete bibliography on this subject.

In this note, we will be interested in the case of rather smooth processes. Namely, consider the family
of processes $\xn$ corresponding to absolutely continuous spectral measures
\[
            F_\nu(du)= \exp\{-|u|^\nu\}du,\qquad 0<\nu<\infty,
\]
and a parallel family of periodic processes $\dxn$ corresponding to discrete spectral measures
 \[
           \tilde F_{\nu}(du)=\sum_{k=-\infty}^\infty   \exp\{-|k|^\nu\}\delta_{2\pi k}
           ,\qquad 0<\nu<\infty.
 \]
The most interesting cases are $\nu=1$ (exponential spectrum) and $\nu=2$ (normal spectrum).

For exposition completeness, let us close the first family with
\[
 F_\infty(du)= \unit_{[-1,1]} du.
\]

Although the smoothness properties of $\xn$ and $\dxn$ are the same, it turns out,
quite surprisingly, that their small deviations behave differently.

An important motivation for this research comes from the recent work of A.W.\ van der Vaart and
J.H.\ van Zanten \cite{VZ}, where such small deviations
were considered in the context of Bayesian statistics. It was shown that they
actually determine posterior convergence rates in nonparametric estimation problems.
In particular the process $X_2$, which is known in the Bayesian and machine learning literature as the
``squared exponential process'', is a popular building block in the construction of
 prior distributions on functional parameters, cf.\ e.g.\ \cite{Ras}.

Before we state the results, let us fix some notation. We write $f(\cdot)\preceq g(\cdot)$ or
$g(\cdot)\succeq f(\cdot)$ if $\limsup \frac {f}{g} <\infty$, while the equivalence $f\approx g$ means that
we have both $f\preceq g$ and $g\preceq f$. Moreover, $f(\cdot)\lesssim g(\cdot)$ or
$g(\cdot)\gtrsim f(\cdot)$ mean that $\limsup \frac {f}{g} \leq 1$. Finally, the strong 
equivalence $f\sim g$ means that $\lim \frac {f}{g}=1$.
\medskip

It was shown in \cite{VZ}, by using the RKHS-entropy method, that
\begin{equation}\label{vzbound}
    \ph(\xn, r) \preceq |\log r|^2, \qquad   \nu\geq 1.
\end{equation}
We will slightly improve this and obtain sharp bounds. Our main results are as follows.

\begin{thm} \label{t1} We have
\be \label{phxn}
\ph(\xn, r) \approx \frac {|\log r|^2}{\log|\log r|}, \qquad   1<\nu\le \infty,
\ee
and
\be \label{phxn1}
\ph(\xn, r) \approx |\log r|^{1+ \frac 1 {\nu}}, \qquad  0<\nu \leq 1,
\ee
as $r\to 0$.
\end{thm}

For the periodic processes the asymptotics is somewhat different.

\begin{thm} \label{t2} We have
\be \label{phdxn}
\ph(\dxn, r) \approx |\log r|^{1+ \frac 1 {\nu}}, \qquad   \nu >0,
\ee
as $r\to 0$.
\end{thm}


\begin{rem} {\rm
The exponential discrete spectrum ($\nu=1$) is well understood for $L_2$-norms where the estimate
\[
    \ph(\dxn, r) \sim C |\log r|^2, \qquad\nu=1,
\]
(and even more precise behavior) is obtained  in the context of small deviations
of the series (with exponentially decreasing coefficients),
see \cite{DLL}, \cite{BorRu}, or \cite{Aur}.   As usual (but not always), the small deviation rate is the same
for the uniform and for the $L_2$-norm.
}\end{rem}

\begin{rem}{\rm
The radical difference of the two bounds (\ref{phxn}) and (\ref{phdxn}) is that the
first one does not depend on $\nu$ while the second one does. From this point of view, Theorem \ref{t1}
provides a more surprising result than Theorem \ref{t2}.

The authors were informed by A.I. Nazarov that the same phenomenon is well known for many years in the theory
of integral operators. Generally speaking, smoother the kernel of a symmetric integral operator is,
faster the eigenvalues decrease. However, there is a kind of barrier: the eigenvalues $\lambda_k$
can not decrease faster than $\log \lambda_k \approx - n\log n$. Since behavior of the eigenvalues
is tightly related to small deviations (once we consider the covariance operator of a Gaussian process)
in $L_2$-norm, the bound for eigenvalues transforms in a bound for small deviations.   
}\end{rem}

\begin{rem}{\rm    Notice that we do not have any general tools
for tracing connections between small deviations for discrete and continuous spectra.
The general feeling is that discrete spectrum provides larger small deviation probabilities.
}\end{rem}

In view of the applications in Bayesian nonparametrics we also provide upper
bounds for the small deviations of rescaled versions of the processes $X_\nu$.
For a constant $c \leq 1$, define the rescaled process $X_\nu^c$ by setting
$X^c_\nu(t) = X_\nu(t/c)$.

\begin{thm} \label{scaling} For all $c \leq 1$ we have
\be\label{eq:s1}
\ph(\xn^c, r) \preceq \frac1c\frac {|\log r|^2}{\log|\log r|}, \qquad   \nu>1,
\ee
\be\label{eq:s2}
\ph(\xn^c, r) \preceq \frac1c |\log r|^{1+ \frac 1 {\nu}}, \qquad  0<\nu \leq 1,
\ee
as $r\to 0$.
\end{thm}

\section{RKHS tools}
\setcounter{equation}{0}

In this section, we recall a powerful approach to the study of Gaussian
small deviations based on the entropy of the corresponding kernel (RKHS),
suggested by J.\ Kuelbs and W.\ Li in \cite{KuLi}. In the literature, this approach
is mainly applied to polynomial entropy, resp. small deviation function,
while the results we need should handle slowly varying functions. Therefore,
for the reader's convenience, we give here the complete proofs.

We work in a fairly general setting. Let $X$ be a centered Gaussian
vector in a separable Banach space $(E,\|\cdot\|)$. Then $X$ generates a
{\it kernel}, or RKHS, $\H$  which is a linear subspace of $E$ equipped
with the structure of a Hilbert space. For a detailed description of the RKHS
we refer to \cite{Lif1}. We denote by $\H_1$ the unit ball of  $\H$. Let
the covering number $N( r )$ be defined as  the minimal number of balls in the norm $\|\cdot\|$ of radius
$ r $ that is needed to cover $\H_1$.  Furthermore,  let $H( r ):=\log N( r )$ be the
corresponding metric entropy of $\H_1$.

We still study the the behavior of small deviation function
\[
\ph(r):=\ph(X, r):=-\log\P(||X||\leq r), \qquad r\to 0.
\]
Let us recall the central inequalities proved in \cite{KuLi}.

\begin{lemma} \label{lem:kl}
Let $ r >0$ and $\lambda>0$. Then
\[
H\left(\frac{2 r }{\lambda}\right) \leq \ph( r ) + \lambda^2/2,
\]
\[
H\left(\frac{ r }{\lambda}\right) \geq \ph(2 r ) +
\log \Phi( \lambda+\alpha_ r ),
\]
where $\Phi$ is the distribution function of the standard normal law,
and $\alpha_ r $ is defined by
$-\log \Phi(\alpha_ r ) = \ph( r )$.
\end{lemma}

We obtain the following corollary from the first inequality in
the case of a slowly varying entropy or small deviation function.

\begin{cor} \label{cor:l1}
Let $\beta$ be any real number  and $\gamma,C>0$. Then
\begin{itemize}
 \item  $\ph( r ) \lesssim C|\log  r |^\gamma (\log|\log r | )^\beta$
        implies
        $H( r ) \lesssim C|\log  r |^\gamma(\log|\log r | )^\beta$.
 \item  $H( r ) \gtrsim C|\log  r |^\gamma(\log|\log r | )^\beta$
        implies
        $\ph( r ) \gtrsim C|\log  r |^\gamma(\log|\log r | )^\beta$.
\end{itemize}
The relations also hold if $\lesssim$ and $\gtrsim$ are replaced
by $\preceq$ and $\succeq$, respectively.
\end{cor}

\begin{proof} Simply set $\lambda=2$ in the first inequality in Lemma~\ref{lem:kl}.
\end{proof}

The arguments are slightly more involved when using the second inequality
because of its implicit nature. First recall that
$$\log \Phi(x) \sim -x^2/2,$$
as $x\to -\infty$.
This helps to simplify the second inequality in Lemma~\ref{lem:kl}.

\begin{lemma}
Let $\lambda=\lambda( r )>0$ be a function such that
$\lambda( r )\leq \sqrt{2 \ph( r )}$. Then, as $ r \to 0$,
\begin{equation} \label{e:kuelbsli2ex}
H\left(\frac{ r }{\lambda}\right) \gtrsim
\ph(2 r ) - \frac{1}{2}\, ( \lambda -\sqrt{2 \ph( r )})^2.
\end{equation}
\end{lemma}

The usual choice in the regularly varying case is
$\lambda=-\alpha_ r  \sim \sqrt{2 \ph( r )}$, which also works in the
case of slow variation. The result reads as follows.

\begin{cor} \label{cor:l2}
Let $\beta$ be any real and $\gamma,C>0$. Then
\begin{itemize}
  \item $H( r ) \lesssim C |\log  r |^\gamma (\log|\log r | )^\beta$
         implies
        $\ph( r ) \lesssim C  |\log  r |^\gamma(\log|\log r | )^\beta$.
  \item Assume that there is a constant $K>0$ such that $\ph( r /2)\leq K \ph( r )$
        for all $ r \in (0,1)$.
        Then \par $\ph( r )\gtrsim C |\log  r |^\gamma(\log|\log r | )^\beta$
        implies \par
        $H( r )\gtrsim C\left(1+\log K/(2 \log 2)\right)^{-\gamma}
        |\log  r |^\gamma(\log|\log r | )^\beta$.
\end{itemize}
The relations also hold if $\lesssim$ and $\gtrsim$ are replaced by $\preceq$ and
$\succeq$, respectively.
 \end{cor}

\begin{proof}
Let $\lambda:=\sqrt{2 \ph( r )}$. For the first implication note that the
assumption for $H$, relation (\ref{e:kuelbsli2ex}), and the fact that
$ r /\lambda\to 0$ imply that
\begin{equation}
C | \log  r  - \log \sqrt{\ph( r )}|^\gamma
(\log|\log r /\sqrt{\ph( r )}| )^\beta \gtrsim \ph(2 r ).
\label{eqn:nnphoemo} \end{equation}

The assumption for $H$ furthermore implies that $H( r )\preceq  r ^{-\tau'}$
for any $\tau'>0$. By Proposition~2.4 in \cite{LiLin}, this yields
\begin{equation}
    \ph( r )\preceq  r ^{-\tau},\qquad
    \text{ for any $\tau>0$.} \label{eqn:phiphoe}
\end{equation}
Thus
$$
\limsup_{ r \to 0} \frac{\log \sqrt{\ph( r )}}{\ph(2  r )^{1/\gamma}}
\leq \frac{\tau}{2} \limsup_{ r \to 0}
\frac{|\log  r |}{\ph(2  r )^{1/\gamma}}.
$$

Also (\ref{eqn:phiphoe}) implies that
$(\log|\log r /\sqrt{\ph( r )}| )^\beta$ can be replaced by
$(\log|\log r | )^\beta$ in (\ref{eqn:nnphoemo}). Therefore
\begin{multline*}
  \frac{1}{C} \leq \liminf_{ r \to 0} \frac{|\log  r
  - \log \sqrt{\ph( r )}|^\gamma}{\ph(2 r )}
  (\log|\log r | )^\beta \\
  = \liminf_{ r \to 0}
  \left|\frac{|\log  r |}{\ph(2 r )^{1/\gamma}}
  + \frac{\log \sqrt{\ph( r )}}{\ph(2 r )^{1/\gamma}}\right|^\gamma
  (\log|\log r | )^\beta
  \leq   \left(1+\frac{\tau}{2}\right)^\gamma
  \liminf_{ r \to 0}  \frac{|\log  r |^\gamma }{\ph(2 r )}
  (\log|\log r | )^\beta.
\end{multline*}
Letting $\tau\to 0$ yields the assertion.

Let us come to the second implication. We may assume that $K\geq 1$.
First note that the regularity assumption on $\ph$ implies that
$\ph( r ) \leq K'  r ^{-h}$ with $h=\log K / \log 2$,
$K':= \ph(1) K$ and all $0< r <1$.
Now if $\ph( r )\gtrsim C |\log  r |^\gamma(\log|\log r | )^\beta$,
we obtain by (\ref{e:kuelbsli2ex}) that
$$
  H\left( \frac{ r }{\lambda}\right)
  \gtrsim C |\log  r |^\gamma (\log|\log r | )^\beta.
$$
We set $ r ':= r /\lambda$. We obtain, by the assumption on $\ph$ that
$ r '\geq  r ^{1+h/2}/\sqrt{2 K'}$. Therefore,
$$
   H\left( r ^{1+h/2} / \sqrt{2 K'}\right) \geq H( r ')
   =H\left( \frac{ r }{\lambda}\right)
   \gtrsim C |\log  r |^\gamma (\log|\log r | )^\beta.
$$
In other words,
$$
   H( r )\gtrsim C |\log  r ^{1/(1+h/2)}|^\gamma (\log|\log r | )^\beta
   =
   \frac{C}{\left(1+\log K/(2 \log 2)\right)^\gamma} \,
   |\log  r |^\gamma (\log|\log r | )^\beta.
$$
\end{proof}

\begin{rem}
{\rm Note that, as in the regularly varying case, one needs to know something
about the maximal increase of $\ph$ in order to translate a lower bound for
$\ph$ into a lower bound for $H$. If it is already known that $\ph$ behaves
logarithmically, then the assumption holds for any $K>1$ and one also obtains
strong asymptotic equivalence.
}
\end{rem}

As a particular case of Corollaries~\ref{cor:l1} and~\ref{cor:l2} we obtain the following.

\begin{cor} \label{cor_ph_eq_H}
Let $\beta$ be any real and $\gamma>0$. Then
$$
   \ph( r ) \approx |\log  r |^\gamma (\log|\log r | )^\beta
   \qquad  \Leftrightarrow\qquad
   H( r )\approx |\log  r |^\gamma (\log|\log r | )^\beta.
$$
\end{cor}

\section{Entropy of stationary RKHS}
\setcounter{equation}{0}

Let now $X$ be a complex valued  centered stationary Gaussian process with spectral
measure $F$ and continuous sample paths. We consider
$X$ as a random element of $E=C[0,1]$. It is well known
(see e.g.\ \cite{Lif1}) that the RKHS $\H$ admits the following representation:
$h\in \H$ iff
\begin{equation} \label{e1}
          h(t)= \int_{-\infty}^{\infty} \ell(u) e^{-itu} F(du),\qquad \ell\in L_2(\R,F),
\end{equation}
for $0\leq t \leq 1$ and
\[
||h||_{\H}=\inf ||\ell||_{2,F}
\]
where infimum is taken over all $\ell$ satisfying (\ref{e1}).
In particular, $h\in\H_1$ (here, as above, $\H_1$ is the unit ball in $\H$) iff (\ref{e1})
holds with $\ell$ such that $||\ell||_{2,F}\leq 1$.

Now we specify this general scheme to the processes we are interested in and evaluate the entropy.

\subsection{Continuous spectrum}
Let now $F(du)=f_\nu(u)du$, where $f_\nu(u)=e^{-|u|^\nu}$, $\nu>0$.
We prove the following.

\begin{prop} \label{entr_cont} For any $\nu >1$ it is true that
\[
    H(\H_1,\eps ) \approx  \frac{|\log\eps|^2}{\log|\log\eps|}\ .
\]
\end{prop}

\begin{proof}

{\it Upper bound.}\  Let $h \in \H_1$. Then the representation (\ref{e1}) holds with some $\ell$
such that
\[
    ||\ell||_{L_2(\R,F)}^2 = \int_{-\infty}^{\infty} |\ell(u)|^2 f_\nu(u) du \leq 1.
\]
Clearly, $h$ turns out to be an entire analytic function well defined on $\C$ by the same
expression (\ref{e1}) and by H\"older's inequality
\[
     |h(z)| \leq \int_{-\infty}^{\infty} e^{|\Imag(z)|\, |u|} |\ell(u)| f_\nu(u) du
            \leq \left( \int_{-\infty}^{\infty} e^{2|\Imag(z)|\, |u|}  f_\nu(u) du \right)^{1/2}:= M_\nu(2|\Imag(z)|).
\]
Since
\[
    \log M_\nu(r) =\frac 12 \ \log  \int_{-\infty}^{\infty} e^{r \, |u|-|u|^{\nu}}  du
    \sim \frac{(\nu-1)\, r^{\nu/(\nu-1)}} {2\, \nu^{\nu/(\nu-1)}}, \qquad \textrm{as} \ r\to\infty,
\]
it follows that
\be \label{c1c2}
   |h(z)| \leq M_\nu(2|\Imag(z)|) \leq C_1\ \exp \{ C_2  |\Imag(z)|^{\nu/(\nu-1)}  \},    \qquad \forall h\in H_1, z\in\C,
\ee
with appropriate constants $C_1=C_1(\nu)$, $C_2=C_2(\nu)$. It is known from Theorem XX of \cite{KT} that the entropy of the
class of all entire analytic functions  $\A(C_1,C_2,\nu)$ satisfying the even weaker condition
\be \label{c1c2_mod}
   |h(z)| \leq C_1\ \exp \{ C_2  |z|^{\nu/(\nu-1)}  \},    \qquad \forall z\in\C,
\ee
verifies
\[
 H(\A(C_1,C_2,\nu),\eps ) \approx \frac{|\log\eps|^2}{\log|\log\eps|}\ .
\]
Since $\H_1\subset\A(C_1,C_2,\nu)$, we obtain
\[
 H(\H_1,\eps ) \preceq \frac{|\log\eps|^2}{\log|\log\eps|}\ .
\]
\medskip

{\it Lower bound.}\ Here we will only need an inequality
\be \label{large_spectr}
   f(u)\geq c_f, \qquad |u|\leq 1,
\ee
for a constant $c_f > 0$,
which is fulfilled for all densities $f_\nu$, $\nu>0$.

We start with a construction of an auxiliary function and study its properties. Take
any $\gamma\in (0,1)$ and let
a sequence $(a_k)_{k\geq 1}$ be defined by $a_k=ck^{-1-\gamma}$ and normalized so that
$\sum_{k=1}^{\infty} a_k =1$. Let
\[
G(z)= \prod_{k=1}^\infty \frac{\sin(a_k z)}{a_k z}, \qquad z\in \C.
\]
Since
\[
\frac{|\sin(z)|}{|z|} \leq \sum_{j=1}^\infty \frac{|z|^{j-1}}{j!} \leq e^{|z|},
\]
we have
\be \label{g_incr}
|G(z)|\leq \exp\left( \sum_{k=1}^{\infty} a_k |z| \right) = e^{|z|}.
\ee
The function $G$ is rapidly decreasing on the real line. Namely, for any large $t\in \R$
choose a positive integer $\kappa=\kappa(t)$ such that $a_\kappa|t|\sim 2$, i.e.
$\kappa\sim(c|t|/2)^{\frac{1}{1+\gamma}}$. Then
\be \label{g_decr}
|G(t)|\leq \prod_{k=1}^\kappa |a_k t|^{-1} \leq 2^{-\kappa}
\leq \exp\left(-C_G |t|^{\frac{1}{1+\gamma}}\right)
\ee
with appropriate $C_G>0$. Finally, notice that
\be \label{g_nonvan}
\theta_G:=\inf_{0\leq t\leq 1} |G(t)|>0,
\ee
since the smallest zero of $G$ is attained at $\frac{\pi}{c}>\pi>1$.

Now we start the entropy estimation. Consider a class $\Psi_K$ of analytic functions $\psi$ on complex plane satisfying
\be \label{exp_bound}
  |\psi(z)| \leq K \exp \{ |z|^{1/2}\}, \qquad z\in \C.
\ee
Again by Theorem XX in \cite{KT} it is true that
\be \label{H_Psi}
   H(\Psi_K,\eps ) \approx \frac{|\log\eps|^2}{\log|\log\eps|}\ .
\ee
Next, consider a class of functions
\[
   B_{K} = \{b: b(z)=\psi(z) G(z),\ \psi \in \Psi_K, z\in \C  \}.
\]
With a minor abuse of notation, we do not distinguish the functions from
$B_{K}$ and their restrictions on $[0,1]$. Clearly,
\be \label{H_B}
     H(B_{K}, \eps ) \geq H(\Psi_K, \theta_G^{-1} \eps ) \succeq \frac{|\log\eps|^2}{\log|\log\eps|} \ .
\ee
We will show now that for an appropriate choice of the parameter $K$ it is true that $B_{K}\subset \H_1$.
Let $b\in B_{K}$.
Then by (\ref{exp_bound}) and (\ref{g_incr}) we have
\[
|b(z)|\leq K \exp\left\{|z|+|z|^{1/2}\right\} .
\]
Moreover, by (\ref{exp_bound}) and (\ref{g_decr})
\[
   ||b||^2_{L_2(\R)} \leq K^2 \int_{-\infty}^{\infty} \exp\left\{2|t|^{1/2}-2C_G |t|^{\frac{1}{1+\gamma}} \right\} dt
   :=K^2 C_{G,2}^2<\infty.
\]
By using these two properties, it follows from the classical Paley--Wiener theorem (\cite{PW} or \cite{Ach}, Chap.\ IV) that
the Fourier transform
\[
\hat b(u)
= \frac{1}{\sqrt{2\pi}}   \int_{-\infty}^{\infty} e^{iut}b(t)dt
= \frac{1}{\sqrt{2\pi}}   \int_{-\infty}^{\infty} e^{iut} \psi(t) G(t)dt
\]
vanishes outside the interval $[-1,1]$.

On the other hand, we can write
\begin{eqnarray*}
b(t) &=& \frac{1}{\sqrt{2\pi}}    \int_{-1}^{1} e^{-iut}\hat b(u)du \\
     &=& \frac{1}{\sqrt{2\pi}}  \int_{-1}^{1} e^{-iut}\ \frac{\hat b(u)}{f(u)}\ f(u) du \\
     &=:&  \int_{-\infty}^{\infty} e^{-iut}\ell(u)\ f(u) du.
\end{eqnarray*}
It remains to show that $||\ell||_{2,F}\leq 1$. By using (\ref{large_spectr}) we have, indeed,
\begin{eqnarray*}
||\ell||_{2,F}^2 &=& \frac{1}{2\pi} \int_{-1}^{1} \frac{|\hat b(u)|^2}{f(u)}\  du
\\
&\leq& \frac{1}{2\pi c_f} ||\hat b||^2_{L_2(\R)} = \frac{1}{2\pi c_f} ||b||^2_{L_2(\R)}
\leq
\frac{K^2 C_{G,2}^2}{2\pi c_f} \leq 1,
\end{eqnarray*}
whenever $K$ is chosen sufficiently small (depending on $c_f$).
Thus $B_{K} \subset \H_1$ and we obtain from (\ref{H_B})
\[
H(\H_1, \eps ) \geq H(B_{K}, \eps )  \succeq \frac{|\log\eps|^2}{\log|\log\eps|} \ ,
\]
as required.
\end{proof}

For small values of $\nu$ we only need the following upper bound.

\begin{prop} \label{entr_cont1}
For any $\nu \leq 1$ it is true that
\[
    H(\H_1,\eps ) \preceq |\log\eps|^{1+1/\nu} .
\]
\end{prop}

\begin{proof} The idea is simple: for $\nu=1$ the result is already known from (\ref{vzbound}) and we
reduce the general case to that one by truncation of the spectral measure. Namely, for any
$\eps>0$ let $v=(3|\log\eps|)^{1/\nu}$. Then by (\ref{e1}) the elements of $\H_1$ have the form
\[
h(t)= \left( \int_{|u|\leq v} +  \int_{|u| > v}  \right) \ell(u) e^{-itu} F(du)
:= h_v(t)+h^v(t), \qquad ||\ell||_{2,F}\leq 1.
\]
By the choice of $v$ we have
\[
|h^v(t)|\leq ||\ell||_{2,F} \left(\int_{|u| > v} \exp(-|u|^\nu) du\right)^{1/2}
\leq C \exp (-|v|^\nu/2) v^{(1-\nu)/2} \leq \eps
\]
for small $\eps$. Therefore, we only need to study the entropy of the set $\H_1^v:=\{h_v, h\in \H_1\}$.
This will be done by means of the following result from \cite{KT} in the quantitative version
of \cite{VZ}, Lemma 2.3.

\begin{lemma}
Let $F$ be a spectral measure and let a positive $\delta<1$ be such that
\[
I:= \int e^{\delta|u|} F(du)\leq 1.
\]
Then
\[
H(\H_1,\eps)\leq C\ \frac{|\log\eps|^2}{\delta}\ ,
\]
where $C$ is a numeric constant.
\end{lemma}
First, notice that if we drop the assumption $I\leq 1$, then by scaling reasons we still
have
\be \label{vzbound2}
H(\H_1,\eps)
\leq C \ \frac{|\log(\eps/\sqrt{I})|^2}{\delta}\ ,
\ee

Apply this bound to our truncated measure $e^{-|u|^\nu}\unit_{|u|\leq v} du$ and
$\delta=\theta |\log\eps|^{1-1/\nu}$ with appropriately small parameter $\theta\leq 3^{-1/\nu}$.
Notice that $\delta|u|\leq |u|^{\nu}$ whenever $|u|\leq v$. Hence
\[
I= \int_{|u|\leq v} e^{\delta|u|-|u|^{\nu}}du \leq 2v \approx |\log\eps|^{1/\nu}.
\]
We obtain from (\ref{vzbound2})
\[
H(\H_1^v,\eps)\preceq \frac{|\log\eps|^2}{|\log\eps|^{1-1/\nu}} =  |\log\eps|^{1+1/\nu},
\]
as required.
\end{proof}

\subsection{Discrete spectrum}
Let now $F(du)=\sum_{k=-\infty}^\infty   \exp\{-|k|^\nu\}\delta_{2\pi k}$,  $\nu>0$.  We prove the following.

\begin{prop} \label{entr_discr} For any $\nu>0$ it is true that
\[
    H(\H_1,\eps ) \preceq  |\log\eps|^{1+1/\nu}.
\]
\end{prop}

\begin{proof}
The reasoning goes along the same lines as that of the upper bound in the previous proposition.
Let $h \in \H_1$. Then the representation (\ref{e1}) means that
\be \label{e1_discr}
    h(t)= \sum_{k=-\infty}^{\infty} \ell_k e^{-ik t-|k|^\nu}
\ee
with some $\ell=(\ell_k)$
such that
\[
    ||\ell||_{L_2(\R,F)}^2 = \sum_{k} |\ell_k|^2   \exp\{-k^\nu\} \leq 1.
\]
Clearly, $h$ turns out to be a periodic entire analytic function well defined on $\C$ by the same
expression (\ref{e1_discr}) and by the H\"older's inequality
\[
     |h(z)| \leq \sum_{k=-\infty}^{\infty} e^{|\Imag(z)|\, |k|- |k|^{\nu}} |\ell_k|
            \leq \left( \sum_{k=-\infty}^{\infty} e^{2|\Imag(z)|\, |k|- |k|^{\nu}} \right)^{1/2}
            := \tilde M_\nu(2|\Imag(z)|).
\]
It follows again that
\[
   |h(z)|  \leq C_1\ \exp \{ C_2 |\Imag(z)|^{\nu/(\nu-1)}  \},
   \qquad \forall h\in H_1, z\in\C,
\]
with appropriate constants $C_1=C_1(\nu)$, $C_2=C_2(\nu)$. It is known by Theorem XXI
in \cite{KT} that the entropy of the
class of all periodic entire analytic functions  $\tilde\A(C_1,C_2,\nu)$ satisfying
this condition verifies
\[
 H(\tilde\A(C_1,C_2,\nu),\eps ) \approx |\log\eps|^{1+1/\nu}.
\]
Hence
\[
   H(\H_1,\eps ) \preceq  |\log\eps|^{1+1/\nu}.
\]
\end{proof}

\section{Proofs of main results}
\setcounter{equation}{0}

\begin{proof} [\ of Theorem \ref {t2}]
The lower bound for small deviations follows immediately from Proposition \ref{entr_discr} and
Corollary \ref{cor:l2}.

For getting the upper bound we implement a simple but ingenious idea of B.S.~Tsirelson
initially designed for continuous spectra in \cite{LifTs}.
Let $l$ be an integer.
Let us consider an auxiliary centered stationary Gaussian process
$Y=Y_l(t)$ with the spectral measure
\[
 F_{Y}(du)=\exp\{-l^\nu\}\ \sum_{|k|\leq l} \delta_{2\pi k},
 \]
which minorates $\tilde F_\nu$. By the standard Anderson argument
\[
\P(||\dxn||_\infty\leq r) \leq
\P(||Y||_\infty\leq r) \qquad \forall r>0.
\]
The covariance of $Y$ is
\be \label{cov}
\E Y(t)Y(0)= \exp\{-l^\nu\}\ \sum_{|k|\leq l}  e^{itk}  = \frac{4 \exp\{-l^\nu\}}{|e^{it}-1|^2}
\, \sin\left(\frac{(2l+1)t}{2}\right)\, \sin\left(\frac{t}{2}\right), \qquad t\not=2\pi k,
\ee
while for the variance we have
\be \label{var}
    \s^2:= \E|Y(t)|^2 = \exp\{-l^\nu\}\ (2l+1).
\ee
 Define a grid step $\D={2\pi}/(2l+1)$. Observe from (\ref{cov})
that $(Y(k\D))_{k\in \Z}$ is a centered Gaussian non-correlated, hence independent, sequence with variance
(\ref{var}). For any $r>0$ we get the bound
\begin{eqnarray*}
\P(||Y||_\infty\leq r) &\leq& \P(\sup_{0\leq k\leq 1/\D} |Y(k\D)|\leq r)\cr
 &\leq& \P(\s|N|\leq r)^{1/\D}
 \leq \left(\sqrt{2/\pi}\ \frac {r}{\s}\right)^{1/\D}
\leq \left(\frac {r}{\s}\right)^{(2l+1)/2\pi}
\cr
&=& \left(\frac {r}{(2l+1) \exp\{-l^\nu\}}\right)^{(2l+1)/2\pi}
\leq \left(r \exp\{l^\nu\}   \right)^{(2l+1)/2\pi}.
\end{eqnarray*}
Next, an elementary optimization suggests to set
\[
l \sim \left(\frac{|\log r|}{\nu+1}\right)^{1/\nu},
\]
whereas
\begin{eqnarray*}
\ph(\dxn,r) &\geq& \ph(Y,r) \geq
- \frac{2l+1}{2\pi}\ \log \left(r \exp\{l^\nu\}\right) \sim
   \frac{l}{\pi}\ (|\log r| - l^\nu)
\\
&\sim&
 \frac{\nu l}{\pi(\nu+1)}\ |\log r|
=  \frac{\nu}{\pi(\nu+1)}\ \frac{|\log r|^{1+ 1/\nu}} {(\nu+1)^{1/\nu}}
\cr
&=&    \frac{\nu}{\pi(\nu+1)^{1+1/\nu}}\ |\log r|^{1+1/\nu},
\end{eqnarray*}
and we arrive at the desired estimate.
\end{proof}


\begin{proof} [\ of Theorem \ref {t1}]

For $\nu >1$
the result follows immediately from Proposition \ref{entr_cont} and Corollary \ref{cor_ph_eq_H}.

For $\nu\leq 1$ the necessary upper bound  follows immediately from
Proposition \ref{entr_cont1} and Corollary \ref{cor:l2}.

The necessary lower bound
\be \label{tsir_cont}
  \ph(\xn,r) \succeq |\log r|^{1+1/\nu}
\ee
holds for any $\nu>0$ and can be obtained by Tsirelson's method, as described above. In this case,
for any positive  $l$ we consider an auxiliary  centered stationary Gaussian process
$Y=Y_l(t)$ with the spectral density
\[
    f_{Y}(u)=\exp\{-l^\nu\}\ \unit_{|u|\leq l},
\]
which minorates $f_\nu$.  Define a grid step $\D=\frac{2\pi}{l}$. It is easy to see again
that $(Y(k\D))_{k\in \Z}$ is a centered Gaussian non-correlated, hence independent, sequence with variance
\[
    \s^2:= \E|Y(t)|^2 = 2 l \exp\{-l^\nu\}.
\]
and the final calculation leading to (\ref{tsir_cont}) goes through exactly as above.
\end{proof}

\begin{rem} {\rm
We see from Theorem \ref{t1} that the estimate (\ref{tsir_cont}) is not sharp for $\nu>1$. This is
rather surprising since in the previously known examples (e.g.\ for polynomially
decreasing spectral densities in \cite{LifTs}) Tsirelson's method always returned the right rates.
}
\end{rem}

\begin{proof} [\ of Theorem \ref {scaling}]

We prove (\ref{eq:s2}), the proof of (\ref{eq:s1}) is identical.
Clearly,
\[
\P(\|X^c_\nu\|_\infty \leq r) = \P\Big(\sup_{t \in [0,1/c]}|X_\nu(t)| \leq r\Big).
\]
Let $\H^c_1$ be the unit ball of the RKHS of the process $X_\nu$ viewed as random element
in $C[0,1/c]$, i.e.\  the class of functions on $[0, 1/c]$
of the form
\[
h(t) = \int_{-\infty}^{\infty} \ell(u)e^{-itu}\,dF_\nu(u), \qquad \|\ell\|_{2,F_\nu} \leq 1.
\]
Let $n$ be the smallest integer larger or equal to $1/c$. Observe that if $h \in \H^c_1$, then for $k = 0, \ldots, n-1$,
the function $t\mapsto h(k+t)$ on $[0,1]$ belongs to the unit ball ${\H}_1$ of the RKHS of the process
$X_\nu$ on $[0,1]$. Hence, if $h_1, \ldots, h_N$ is
an $\eps$-net for $\H_1$, then the functions of the form
\[
     t \mapsto \sum_{k=0}^{n-1} h_{j_k}(t-k) \unit_{[k, k+1)}(t)
\]
form an $\eps$-net for $\H^c_1$. There are at most $N^n$ such functions. We keep only those  for which 
there exists an element of $\H^c_1$ at
uniform distance at most $\eps$. The mentioned elements form a $2\eps$-net for $\H^c_1$.  
It follows that
\[
H(\H^c_1, 2\eps) \leq n H({\H}_{1}, \eps) \leq \frac 2 c \ H({\H}_{1}, \eps).
\]
Now apply Proposition \ref{entr_discr} and Corollary \ref{cor:l2} to arrive at (\ref{eq:s2}).
\end{proof}

\section{An open problem}
\setcounter{equation}{0}

It would be very interesting to extend our results to more general classes of smooth processes. Since
Tsirelson's bound is sharp for spectral measures $F_\nu$, $0<\nu\le 1$, and in the case of polynomial 
spectral density $f(u)\approx |u|^{-1-\beta}$ it is also known to give a sharp bound
$\ph(X,r) \approx r^{-2/\beta}$ , 
it is natural to conjecture that this bound is sharp in all intermediate cases, too. Our methods provide some 
reasonable bounds for general case but they should be at least enhanced in order to solve it properly. 
For example, on the test family of intermediate processes $Y_\alpha$  with spectral densities
\[ 
   f_\alpha(u) = \exp\{- (\log_+|u|)^\alpha\} ,\qquad \alpha>1,
\]
we get 
\[
|\log r|^{\frac{\alpha-1}{\alpha}}\ \exp\left\{ (2|\log r|)^{1/\alpha} \right\}
\preceq
\ph(Y_\alpha,r) 
\preceq
|\log r|\ \exp\left\{ (2|\log r|)^{1/\alpha}  + \frac{5}{\alpha}|\log r|^{2/\alpha-1}  \right\},
\]
which is not as sharp as we would like.
\end{document}